\documentclass[11pt]{amsart}
\usepackage{mathrsfs,amssymb}
\usepackage{verbatim}

\usepackage{graphicx}
\begin{document}

\newtheorem{theorem}{Theorem}[section]
\newtheorem{proposition}[theorem]{Proposition}
\newtheorem{lemma}[theorem]{Lemma}
\newtheorem{corollary}[theorem]{Corollary}
\newtheorem{conjecture}[theorem]{Conjecture}
\newtheorem{question}[theorem]{Question}
\newtheorem{problem}[theorem]{Problem}
\theoremstyle{definition}
\newtheorem{definition}{Definition}

\theoremstyle{remark}
\newtheorem{remark}[theorem]{Remark}

\def\theenumi{\roman{enumi}}

\numberwithin{equation}{section}

\renewcommand{\Re}{\operatorname{Re}}
\renewcommand{\Im}{\operatorname{Im}}

\def \G {{\Gamma}}
\def \g {{\gamma}}
\def \R {{\mathbb R}}
\def \HH {{\mathbb H}}
\def \C {{\mathbb C}}
\def \Z {{\mathbb Z}}
\def \Q {{\mathbb Q}}
\def \TT {{\mathbb T}}
\newcommand{\T}{\mathbb T}
\def \vol {\hbox{vol}}

\newcommand{\tr}[1] {\hbox{tr}\left( #1\right)}

\newcommand{\area}{\operatorname{area}}

\newcommand{\Op}{\operatorname{Op}}
\newcommand{\dom}{\operatorname{Dom}}
\newcommand{\Dom}{\operatorname{Dom}}

\newcommand{\Norm}{\mathcal N}
\newcommand{\simgeq}{\gtrsim}%
\newcommand{\simleq}{\lesssim}

\newcommand{\UN}{U_N}
\newcommand{\OPN}{\operatorname{Op}_N}
\newcommand{\HN}{\mathcal H_N}
\newcommand{\TN}{T_N}  
\newcommand{\PDO}{\Psi\mbox{DO}}

\newcommand{\oi}{\mathcal I}

\newcommand{\length}{\operatorname{length}}

\newcommand{\curve}{\mathcal C} 
\newcommand{\vE}{\mathcal E} 

\newcommand{\dist}{\operatorname{dist}}
\newcommand{\supp}{\operatorname{supp}}
\newcommand{\spec}{\operatorname{spec}}
\newcommand{\diam}{\operatorname{diam}}

\newcommand{\Ccap}{\operatorname{Cap}}

\newcommand{\paramone}{C_1}
\newcommand{\prho}{\sigma}
\newcommand{\tg}{g}
\newcommand{\sumstar}{\sideset{}{^\ast}\sum}

\title{Nodal intersections and $L^p$ restriction theorems on the torus}
\author{ Jean Bourgain and Ze\'ev Rudnick}
\address{School of Mathematics, Institute for Advanced Study,
Princeton, NJ 08540 } \email{bourgain@ias.edu}

\address{Raymond and Beverly Sackler School of Mathematical Sciences,
Tel Aviv University, Tel Aviv 69978, Israel}
\email{rudnick@post.tau.ac.il}

\thanks{J.B. thanks the UC Berkeley
math department for their hospitality. J.B. was partially funded by
NSF grants DMS-0808042 and DMS-0835373. The research leading to
these results has received funding from the European Research
Council under the European Union's Seventh Framework Programme
(FP7/2007-2013) / ERC grant agreement n$^{\text{o}}$ 320755. }

\date{January 28, 2014}
\maketitle

\section{Introduction}

\subsection{Nodal intersections}
 Let $\curve\subset \TT^2$ be a curve on the standard torus
$\TT^2=\R^2/2\pi\Z^2$ , which has nowhere-zero curvature. Let $F$ be
a real-valued eigenfunction of the Laplacian on $\TT^2$ with
eigenvalue $\lambda^2$: $-\Delta F =\lambda^2 F$. We want to
estimate the number of nodal intersections
\begin{equation}
N_{F,\curve}=\#\{x: F(x)=0\} \cap \curve \end{equation}
 that is the number of zeros of $F$ on $\curve$ .


If $\curve$ is real analytic, then upper bounds of the form
$N_{F,\curve}\ll \lambda$ can be obtained from a result of Toth and
Zelditch \cite{TZ} (see also \cite{BRGAFA}, \cite{ET}) once we have
an exponential restriction lower bound $\int_{\curve}|F|^2\gg e^{-c
{\lambda }} ||F||_2^2$ for the $L^2$-norm of $F$ restricted to
$\curve$, in terms of the $L^2$-norm $||F||_2^2 =
\int_{\TT^2}|F(x)|^2dx$.
In the case of the torus, for any smooth $\curve$ with non-vanishing
curvature we have earlier obtained a {\em uniform} $L^2$-restriction
bound \cite{BRCRAS}
\begin{equation}\label{lower bound BR}
   \int_{\curve} |F|^2
\gg ||F||_2^2
\end{equation}
(the implied constants depending only on the curve $\curve$)  and
hence by \cite{TZ} we get an upper bound for $\curve$ analytic
\begin{equation}\label{upper bd on N}
N_{F,\curve} \ll \lambda
\end{equation}
In our paper \cite{BRGAFA} we also obtained a lower bound for
$N_{F,\curve}$ when the curve $\curve$ has non-vanishing curvature:
\begin{equation}
N_{F,\curve} \gg \lambda^{1-o(1)}
\end{equation}
We conjecture that the correct lower bound is
\begin{equation}\label{conj N}
N_{F,\curve}\gg \lambda
\end{equation}
 that is the lower bound should be the same order of magnitude as the upper bound.

In this paper we approach conjecture \eqref{conj N} by giving a
lower bound for $N_{F,\curve}$ in terms of an arithmetic quantity,
the maximal number $B_\lambda$ of lattice points which lie on an arc
of size $\sqrt{\lambda}$ on the circle $|x|=\lambda$:
\begin{equation}\label{def of B}
B_\lambda=\max_{|x|=\lambda} \#\{\xi\in \vE: |x-\xi|\leq
\sqrt{\lambda}\}
\end{equation}
where  $\vE = \vE_\lambda$ is the set of all lattice points on the
circle $|x|=\lambda$.

\begin{theorem}\label{thm lower bd for N}
If $\curve$ is smooth with non-zero curvature then
 \begin{equation}
 N_{F,\curve} \gg \lambda /B_\lambda^{5/2}
 \end{equation}
 \end{theorem}
 According to the conjecture of Cilleruelo and Granville \cite{CG},
$B_\lambda=O(1)$ is bounded, which in view of Theorem~\ref{thm lower
bd for N}, implies conjecture \eqref{conj N}.

The conjecture of Cilleruelo and Granville  is known for "almost
all" $\lambda$ \cite{BRAHP}, but individually we only know a bound
of $B_\lambda\ll \log \lambda$, see \S~\ref{sec:lattice pts} .

To contrast with these results, we show in \S~\ref{Sec geod}  that
no lower bounds for $N_{F,\curve}$ are possible when the curvature
is zero, that is for geodesic segments, in fact that
$\liminf_\lambda N_{F,\curve}=0$. We also briefly discuss the
situation on the sphere.

\subsection{Relation with $L^p$ restriction theorems}
To prove Theorem~\ref{thm lower bd for N} we start by giving a lower
bound for $N_{F,\curve}(\lambda)$ in terms of a lower bound for the
restriction $L^1$-norm:
In \S~\ref{sec: of of thm on N} we show
\begin{theorem}\label{thm N in terms of m int}
If $\curve$ is smooth with non-zero curvature then
\begin{equation}
 N_{F,\curve} \gg  \lambda \cdot  \left( \frac 1{||F||_2} \int_\curve |F|\right)^5
\end{equation}
\end{theorem}
We conjecture a uniform lower bound for the restriction $L^1$-norm,
which will imply \eqref{conj N}.

%
Next, in \S~\ref{sec:reduction to L4} we give a lower bound for
$||F||_{L^1(\curve)}=\int_\curve|F|$ in terms of the restriction
$L^4$ norm:
\begin{equation}\label{L1 in terms of L4 v2 int}
||F||_{L^1(\curve)}\gg_\curve \frac {||F||_2^3}
{||F||^{2}_{L^4(\curve)}}
\end{equation}
Thus we find that we are reduced to giving an {\em upper} bound on
the restriction $L^4$-norm.  In \S~\ref{sec: the L4 norm} we show
\begin{theorem}\label{thm L4 int}
If $\curve$ is smooth with non-zero curvature then
\begin{equation}
||F||_{L^4(\curve)}  
\ll B_\lambda^{1/4} ||F||_2
\end{equation}
\end{theorem}
Inserting Theorem~\ref{thm L4 int} into  \eqref{L1 in terms of L4 v2
int} we obtain
\begin{equation}
\frac 1{||F||_2}\int_\curve |F| \gg \frac{1}{\sqrt{ B_\lambda}}
\end{equation}
and using Theorem~\ref{thm N in terms of m int} we obtain
Theorem~\ref{thm lower bd for N}.

\subsection{Prior results}
There are very few  lower bounds on the number of nodal
intersections available for other models.  In the case of the
modular domain $\HH^2/SL_2(\Z)$ and $\curve$ being a closed
horocycle, Ghosh, Reznikov and Sarnak \cite{GRS} give a lower bound
$N_{F,\curve}\gg \lambda^{1/12-o(1)}$ for eigenfunctions $F$ which
are joint eigenfunctions of all Hecke operators, and assuming the
Generalized Riemann Hypothesis  they give a similar result when
$\curve$ is a sufficiently long segment of the infinite geodesic
running between two cusps,

Concerning upper bounds,   El-Hajj and Toth  \cite{ET} show that for
a bounded, piecewise-analytic convex domain with ergodic billiard
flow and $\curve$   an analytic interior curve with strictly
positive geodesic curvature, the upper bound \eqref{upper bd on N}
holds for a density-one subsequence of eigenfunctions. For
eigenfunctions on a compact hyperbolic surface, Jung \cite{Jung} has
recently obtained an upper bound analogous to \eqref{upper bd on N}
when $\curve$ is a geodesic circle.


\section{Lattice points and geometry}\label{sec:lattice pts}

\subsection{Lattice points in short arcs}

We denote by $\vE = \vE_\lambda$ the set of lattice points on the
circle $|x|=\lambda$. As is well known, $\#\vE \ll \lambda^{o(1)}$
and can grow faster than any power of $\log \lambda$. Concerning
lattice points in short arcs, Jarnik \cite{Jarnik} showed that any
arc of length $\lambda^{1/3}$ contains at most two lattice points.
Cilleruelo and C\'ordoba \cite{CC} showed that for fixed $\delta>0$,
any arc of length $\lambda^{1/2-\delta}$ contains at most
$M(\delta)$ lattice points. The natural conjecture here \cite{CG} is
that the same statement holds for arcs of length
$\lambda^{1-\delta}$. However this is still open even for arcs of
size $\sqrt{\lambda}$. That turns out to be a critical regime for
us, and we set
\begin{equation}
B_\lambda = \max_{|x|=\lambda} \#\{\mu \in \vE:
|\mu-x|<\sqrt{\lambda} \}
\end{equation}
to be the maximal number of lattice points in arcs of size
$\sqrt{\lambda}$.

\begin{lemma}
Let $B=B_\lambda(c)$ be the maximal number of lattice points of
$\vE$ in an arc of length $c\sqrt{\lambda}$, $c\geq 1/2$. Then
\begin{equation}
B\ll  c\log \lambda
\end{equation}
\end{lemma}
\begin{proof}
To see this, we recall that Cilleruelo and C\'ordoba \cite{CC}
showed that if  $P_1,\dots,P_m \in \vE$ are  distinct lattice points
on the circle of radius $\lambda$, then
\begin{equation}\label{CC ineq}
 \prod_{1\leq i< j\leq m} |P_i-P_j| \geq \lambda^{e(m)},
\quad e(m) = \begin{cases} \frac m2(\frac m2-1),& m \mbox{ even}\\
\frac 14(m-1)^2,& m \mbox{ odd}
             \end{cases}
\end{equation}

Thus if $P_1,\dots,P_m \in \vE$  lie in an arc of diameter $D=\max
|x-y|<\sqrt{\lambda}/2$, then by \eqref{CC ineq} we find
\begin{equation}
  D^{m(m-1)/2} \geq \lambda^{m(m-2)/4}
\end{equation}
and hence
\begin{equation}\label{slight arc}
m\leq 
\frac{\log \lambda}{2\log 2}+1 \ll \log \lambda
\end{equation}
Now for an arc of length $c\sqrt{\lambda}$, $c>1/2$, divide it into
$\approx 2c$ smaller arcs of length $\sqrt{\lambda}/2$ and use
\eqref{slight arc} to find that it contains  $\ll c\log \lambda $
lattice points.
\end{proof}

\subsection{Medians}\label{sec:medians}

Given a pair of   points $\mu$, $\nu$ on the circle $|x|=\lambda$,
their median is $z=\frac 12(\mu+\nu)$. This gives a map from pairs
of  points on the circle $\lambda S^1$ to points in the disc of
radius $\lambda$:
$$z: \lambda S^1 \times \lambda S^1
\to \{|z|\leq \lambda\}$$ By definition, if $\mu=\nu$ then $z=\mu$.
Note that the origin is the median of all pairs of antipodal points
$\{\mu,-\mu\}$.


Conversely, given a nonzero point in the interior of the punctured
disk $\{0<|x|< \lambda\}$, we can display it as the median of a
unique (unordered) pair of points obtained as the intersection of
the circle $\lambda S^1$ with line through $z$ perpendicular to the
radial line between $z$ and the origin, see Figure~\ref{fig:median}.
\begin{figure}[h]
\includegraphics[width=0.4\textwidth]{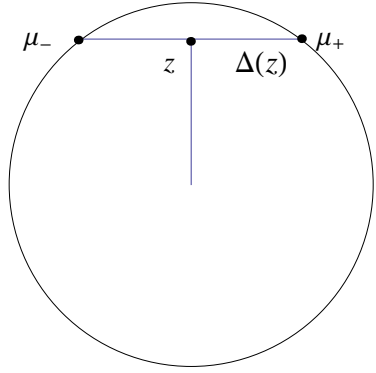}%
\caption{The median map and its inverse} \label{fig:median}
\end{figure}
In fact the formula for these points is
\begin{equation}\label{eq for mu}
\mu_\pm(z) =z \pm \Delta(z) \frac{z^\perp}{|z^\perp|}
\end{equation}
where if $z=(x,y)$ then $z^\perp = (-y,x)$, and
where we set (see Figure~\ref{fig:median})
\begin{equation}
\Delta(z) = \sqrt{\lambda^2-|z|^2}= \frac 12 |\mu_+(z)-\mu_-(z)|
\end{equation}

Let $\mathcal Z=\mathcal Z_\lambda  $ be the set of medians of integer points
with $|\mu|=\lambda$. 
Note that $\#\mathcal Z\leq (\#\vE)^2\ll \lambda^{o(1)}$.

\begin{lemma}\label{stability lemma}
Given  vectors $z,v\in \R^2$, the number of $w\in \mathcal Z_\lambda$
for which
\begin{equation}\label{eq: restrict mu}
|\mu_+(w)-v|<\sqrt{\lambda}
\end{equation}
and
\begin{equation}
|w-z|<\lambda^{1/3}
\end{equation}
 is at most $O(B_\lambda)$.
\end{lemma}
\begin{proof}
The medians $w $ satisfying \eqref{eq: restrict mu} have their
corresponding lattice points $\mu_+(w)$ each lying in an arc of
length about $ \sqrt{\lambda}$, and hence there are at most
$B_\lambda$ possibilities for $\mu_+(w)$. 
Given $\mu_+(w)$, we have at most $2$ possibilities for $\mu_-(w)$:
Indeed, since $w=(\mu_+(w)+\mu_-(w))/2$, we have
\begin{equation}
\mu_-(w) = 2w-\mu_+(w) = 2z-\mu_+(w) +2(w-z)
\end{equation}
Since $|w-z|<\lambda^{1/3}$, given $z$ and $\mu_+(w)$ we know
$\mu_-(w)$ up to an error of  $O(\lambda^{1/3})$; by Jarnik's
theorem, which states that an arc of size $\lambda^{1/3}$ contains
at most two lattice points, this implies there are at most two
possibilities for $\mu_-(w)$.

Since $w=(\mu_+(w)+\mu_-(w))/2$ is determined by knowing both
$\mu_\pm(w)$, we see that there are at most $O(B_\lambda)$
possibilities for $w$.
 \end{proof}


\section{An oscillatory integral along the curve}

\subsection{Phase functions on the curve}

Let $\TT^2 = \R^2/2\pi \Z^2$ be the standard flat torus. An
eigenfunction $F$ of the Laplacian on $\TT^2$ with eigenvalue
$\lambda^2$ has a Fourier expansion
\begin{equation}
F(x) = \sum_{\mu\in \vE} a_\mu e^{i\langle \mu , x\rangle }
\end{equation}
For $F$ to be real valued forces $\overline{a_\mu} = a_{-\mu}$.
The supremum of $F$ is bounded by
\begin{equation}\label{Linfty}
||F||_\infty \leq \sum_{\mu\in \vE} |a_\mu| \leq \frac 1{2\pi} ||F||_2\sqrt{\#\vE}
\end{equation}
We normalize so that
\begin{equation}
4\pi^2 ||F||_2^2 = \sum_{\mu\in \vE} |a_\mu|^2 = 1
\end{equation}

Let $\gamma:[0,L]\to \curve$ be an arc-length parameterization of
$\curve$, so that $\gamma'(t)$ is the unit tangent vector to the
curve at the point $\gamma(t)$. Denote by $n(t)$ the standard unit
normal to the curve at the point $\gamma(t)$, so that $ \gamma''(t)
= \kappa(t)n(t)$ with $\kappa(t)$ the curvature. Let $K_{\min}>0$
and $K_{\max}$ be the minimum and maximum values of the curvature,
so that
\begin{equation}\label{pinching}
0<K_{\min}\leq\kappa(t)\leq K_{\max}
\end{equation}
By shrinking the curve $\curve$, we may assume that its total
curvature is $<\pi/2$.

 We denote $f(t)=F( \gamma(t))$. Using the
Fourier expansion of $F$, we write
\begin{equation}
f(t) = \sum_{\mu\in \vE} a_\mu e^{i\langle \mu , \gamma(t)\rangle }
=   \sum_{\mu\in \vE} a_\mu e^{i \lambda \phi_\mu(t)}
\end{equation}
where the phase function $\phi_\mu$ is
\begin{equation}
\phi_\mu(t) = \langle \frac{\mu}{|\mu|},\gamma(t)\rangle
\end{equation}
The derivative of $\phi_\mu$ is
\begin{equation}
 \phi'_\mu(t) = \langle \frac{\mu}{|\mu|}, \gamma'(t)\rangle =
\sin \alpha_\mu(t)
\end{equation}
where $\alpha_\mu(t)$ is the angle between the normal vector $n(t)$
and $\mu$. Since we assume the total curvature of the curve is
$<\pi/2$, the change in the angle $\alpha_\mu$ is less than $\pi/2$.
The second derivative is
\begin{equation}
 \phi''_\mu(t) = \langle \frac{\mu}{|\mu|}, \gamma''(t)\rangle =
\kappa(t) \cos \alpha_\mu(t)
\end{equation}
By \eqref{pinching},
\begin{equation}\label{bound on phi''}
| \phi''_\mu|\leq K_{\max}
\end{equation}
The third derivative is
\begin{equation}
\phi_\mu''' = \langle \frac\mu{|\mu|},\gamma''' \rangle = \langle
\frac\mu{|\mu|}, \kappa'n + \kappa n' \rangle =\langle
\frac\mu{|\mu|}, \kappa'n- \kappa^2 \gamma' \rangle
\end{equation}
(since $n'=-\kappa \gamma'$) and hence, since $n\perp \gamma'$,
\begin{equation}\label{bound on phi'''}
|\phi_\mu'''|\leq  (|\kappa'|_\infty^2+K_{\max}^4)^{1/2}
\end{equation}
is bounded independent of $\mu$.

\begin{lemma}\label{lem length Bmu}
For $0<\prho\ll 1$ sufficiently small, let $B_\mu$ be the set of points  where $|
\phi'_\mu(t) |<2\prho$. Then $B_\mu$  is an interval and
\begin{equation}\label{length Bmu}
\length B_\mu\ll \frac {\prho}{K_{\min}}
\end{equation}
\end{lemma}

\begin{proof}
We have
\begin{equation}
| \phi'_\mu(t) |= |\langle \frac{\mu}{|\mu|}, \gamma'(t)\rangle |
 = |\sin \alpha_\mu(t)| < 2\prho
\end{equation}
Since we assume the total curvature is $<\pi/2 $,  the change in the
angle $\alpha_\mu$ is less than $\pi/2$ and hence $B_\mu = (c_,c_+)$
consists of at most a single interval.


Since $B_\mu $ is in particular connected  and
$$|\phi''_\mu(t)| =
\kappa(t)|\cos \alpha_\mu| \geq K_{\min}\sqrt{1-4\prho^2}\geq
K_{\min}/2$$
on $B_\mu$, we may  assume that $\phi''_\mu\geq
K_{\min}/2>0$ on $B_\mu$ so that $\phi'_\mu$ is monotonically
increasing. Then $\phi'_\mu(c_-)\geq -2\prho$, $\phi'_\mu(c_+)\leq
+2\prho$ and we have
\begin{equation}
4\prho \geq  \phi_\mu'(c_+)-\phi_\mu'(c_-) = (c_+-c_-)\phi''_\mu(c)
\end{equation}
for some $c\in (c_-,c_+)$ and hence
\begin{equation}
\length B_\mu =  c_+-c_- \leq  \frac{4\prho}{\phi''_\mu(c)} \leq
\frac{8\prho}{K_{\min}}
\end{equation}
as claimed.
\end{proof}

\subsection{Van der Corput's lemma}
Let $[a,b]$ be a finite interval,  $\phi\in C^\infty[a,b]$  a smooth
and real valued phase function,  and $A\in C^\infty[a,b]$ a smooth
amplitude. For $\lambda>0$ define the oscillatory integral
\begin{equation}
I(\lambda):=\int_a^b A(t) e^{i\lambda \phi(t)} dt
\end{equation}
We will need the following well-known result, due to van der Corput
(see e.g. \cite{Stein})
\begin{lemma}\label{vdClem}
Assume that $|\phi''|\geq 1$. Then
\begin{equation}\label{ineq stationary}
|I(\lambda)|\ll \frac 1{\lambda^{1/2}} \left\{||A||_\infty +
||A'||_1 \right\}
\end{equation}
If $|\phi'|\geq 1$ and moreover $\phi'$ is monotonic then
\begin{equation}\label{ineq nonstationary}
|I(\lambda)|\ll \frac 1{\lambda} \left\{||A||_\infty + ||A'||_1
\right\}
\end{equation}
the implied constants absolute.
\end{lemma}

\subsection{An  oscillatory integral along a curve}
For each $0\neq \xi \in \R^2$ define a phase function on the curve
$\curve$ by
\begin{equation}
\phi_\xi(t) = \langle \frac{\xi}{|\xi|},\gamma(t) \rangle
\end{equation}

Let $A\in C^\infty[0,L]$ be a smooth  amplitude, $k$ real  and
\begin{equation}
I(k) = \int A(t)e^{i k \phi_\xi(t)}dt \;.
\end{equation}
\begin{lemma} \label{osc int curve}
For $|k|\geq 1$,
\begin{equation}
|I(k)|\ll \frac 1{|k|^{1/2}}\left\{||A||_\infty + ||A'||_1 \right\}\;,
\end{equation}
the implied constant depending only on the curve $\curve$
(independent of $\xi$).
\end{lemma}
\begin{proof}
We wish to apply Lemma~\ref{vdClem}.  Since  the total curvature of
$\curve$ is $<\pi/2$, each of the phase functions $\phi_\xi$ has at
most one stationary point (at a point where $\xi$ is normal to the
curve). Moreover $\phi_\xi''(t) = \kappa(t)\cos\alpha_\xi(t)$  has
at most one sign change since we restrict the total curvature to be
$<\pi/2$.

Near a stationary point $t_0$, we
have $|\phi'_\xi(t)|<1/2$ if $|t-t_0|<1/(2K_{\max})$, since
\begin{equation}
|\phi_\xi'(t)| = |\phi_\xi'(t)-\phi_\xi'(t_0)| = |t-t_0| \cdot |\phi_\xi''(t_1)|\leq
K_{\max}|t-t_0|
\end{equation}
If $|\phi'_\xi(t)|<1/2$, then
\begin{equation}
|\phi_\xi''(t)| = \kappa(t) |\cos \alpha_\xi(t) |  =
\kappa(t)\sqrt{1-\phi_\xi'(t)^2} \geq K_{\min}\frac{\sqrt{3}}2
\end{equation}

Hence we may cut the curve, that is the arc-length parameter
interval $[0,L]$, into at most $4$ segments on each of which either
$|\phi_\xi''|\geq \frac{\sqrt{3}}2 K_{\min}>0$ or $|\phi_\xi'|\geq
1/2$ and $\phi_\xi''$ does not change sign, hence $\phi_\xi'$ is
monotonic. Then we can invoke Lemma~\ref{vdClem} to deduce that
either \eqref{ineq nonstationary} or \eqref{ineq stationary} hold,
and since $|k|\geq 1$ we have \eqref{ineq stationary} valid in both
cases.
\end{proof}


\section{A bilinear inequality on the curve}
As before let $\vE = \{\mu\in \Z^2:|\mu|=\lambda\}$. For each
$\mu\in \vE$ let $h_\mu(t)\in C_c^1(\R)$ and $a_\mu \in \C$ with
$\sum_{\mu\in \vE} |a_\mu|^2 = 1$. Let
\begin{equation}
H(t) := \sum_{\mu\in \vE} a_\mu h_\mu(t)e^{i\langle \mu, \gamma(t)
\rangle}
\end{equation}
\begin{lemma}\label{bilinear}
\begin{equation}
||H||_2^2 \leq 2 \max_{\mu\in \vE} ||h_\mu||_2^2  + O( \frac
{\#\vE}{\lambda^{1/6}} \left\{ \max_{\mu\in \vE} ||h_\mu||_\infty^2
+ \max_{\mu\in \vE}||h_\mu||_\infty \max_{\mu\in
\vE}||h_\mu'||_1\right\})
\end{equation}
\end{lemma}

\begin{proof}
Multiplying out gives
\begin{equation}\label{double sum}
||H||_2^2 = \sum_{\mu,\nu} a_\mu \bar a_\nu \int h_\mu(t)
\overline{h_\nu(t)} e^{i|\mu-\nu|\phi_{\mu-\nu}(t)}dt
\end{equation}
where if $\mu=\nu$ we set $\phi_0(t) \equiv 1$.
We separate the double sum  \eqref{double sum} to  a sum over "close" pairs
$(\mu,\nu)$, that is such that $|\mu-\nu|<\lambda^{1/3}$, and to a
sum over  the remaining "distant" pairs. We claim that the "close"
pairs contribute
\begin{equation}\label{close pairs1}
\mbox{close }  \leq 2 \max_{\mu\in
\vE} ||h_\mu||_2^2
\end{equation}
while the "distant" pairs contribute at most
\begin{equation}
\mbox{distant } \ll \frac {\#\vE}{\lambda^{1/6}} \left\{
\max_{\mu\in \vE} ||h_\mu||_\infty^2 + \max_{\mu\in
\vE}||h_\mu||_\infty \max_{\mu\in \vE}||h_\mu'||_1\right\}
\end{equation}
\subsubsection{Close pairs}
Given $\mu \in \vE$, certainly we can take $\nu=\mu$ to get a
"close" pair. By Jarnik's theorem \cite{Jarnik}, given $\mu\in \vE$
there is at most one other element of $\vE$ at distance $\leq
\lambda^{1/3}$ from $\mu$, call it $\tilde \mu$ (if it exists).
Estimating the integral trivially by
\begin{equation}
\left|\int h_\mu(t) \overline{h_\nu(t)} e^{i\langle
\mu-\nu,\gamma(t)\rangle} dt  \right| \leq ||h_\mu||_2 \cdot
||h_\nu||_2
\end{equation}
we find  that the contribution of "close" pairs is bounded by
\begin{equation}
\max_\mu   ||h_\mu||_2^2 \cdot \sum_{\mu\in \vE} |a_\mu|^2 +
|a_\mu||a_{\tilde \mu}|
\end{equation}
If $\mu$ does not have a close neighbor other than itself, the term
$a_\mu a_{\tilde \mu}$ is zero. Otherwise, use $|a_\mu a_{\tilde
\mu}|\leq \frac 12 ( |a_\mu|^2 + |a_{\tilde \mu}|^2)$.  Since each
$\mu$ has at most one such close neighbor $\tilde \mu$, the sum over
all $\mu \in \vE$ is at most
\begin{equation}
\sum_{\mu\in \vE}   |a_\mu||a_{\tilde \mu}| \leq \sum_{\mu\in
\vE}\frac 12 ( |a_\mu|^2 + |a_{\tilde \mu}|^2) \leq \sum_{\mu\in
\vE} |a_\mu|^2 = 1
\end{equation}
and hence
\begin{equation}
\mbox{close }\leq 2\max_\mu   ||h_\mu||_2^2
\end{equation}

\subsubsection{Distant pairs}
We now bound the contribution of pairs $\mu,\nu$ with
$|\mu-\nu|>\lambda^{1/3}$ by
\begin{equation}
\mbox{distant }\leq \sum_{|\mu-\nu|>\lambda^{1/3}} |a_\mu||a_\nu|
|I(\mu,\nu)|
\end{equation}
where
\begin{equation}
I(\mu,\nu):= \int h_\mu(t)\overline{h_\nu(t)} e^{i\langle
\mu-\nu,\gamma(t)\rangle} dt = \int
A_{\mu,\nu}(t)e^{i|\mu-\nu|\phi_{\mu-\nu}(t)}dt
\end{equation}
where $A_{\mu,\nu} =h_\mu(t)\overline{h_\nu(t)}$.

By Lemma~\ref{osc int curve},
\begin{equation}\label{bound for I(mu,nu)}
\begin{split}
|I(\mu,\nu)|& \ll \frac 1{|\mu-\nu|^{1/2}} (||A_{\mu,\nu}||_\infty
+||A'_{\mu,\nu}||_1) \\ & \ll \frac 1{\lambda^{1/6}}\left\{
\max_{\mu\in \vE} ||h_\mu||_\infty^2 + \max_{\mu\in
\vE}||h_\mu||_\infty \max_{\mu\in \vE}||h'_\mu||_1\right\}
\end{split}
\end{equation}


Using \eqref{bound for I(mu,nu)}  and $\sum_{\mu,\nu\in  \vE} |a_\mu
a_\nu|\leq \#\vE \sum_\mu |a_\mu|^2=\#\vE$ we find that
\begin{equation}
\begin{split}
\mbox{distant } &\leq \sum_{\mu,\nu}|a_\mu| |a_\nu|
\max_{|\mu-\nu|>\lambda^{1/3} } I(\mu,\nu) \\
& \ll \frac{\#\vE}{\lambda^{1/6}}\left\{ \max_{\mu\in \vE}
||h_\mu||_\infty^2 + \max_{\mu\in \vE}||h_\mu||_\infty \max_{\mu\in
\vE}||h'_\mu||_1\right\}
\end{split}
\end{equation}
as claimed.
\end{proof}


\section{Proof of Theorem~\ref{thm N in terms of m int}}\label{sec: of of thm on N}
\subsection{Overview}
We denote $f(t)=F( \gamma(t))$, which is real valued,  and want to count zeros of $f$ on
$[0,L]$.   The idea is to detect sign changes of $f$ by comparing
$\int  |f|$ and $|\int f |$.

 Let $\paramone$ be a parameter, which we will want to satsify $1\ll \paramone  = o(\lambda)$, and consider a
partition of unity $\{\tau_j\}_{j\in J}$ of the interval $[0,L]$,
where $\tau_j \geq 0$, $\sum_j \tau_j = \mathbf 1_{[0,L]}$, so that
\begin{enumerate}
\item $\#J\approx \lambda/\paramone$
\item  $\tau_j$ supported in an interval of length $\approx
\paramone/\lambda$
\item
$ | \partial^r \tau_j/\partial t^r | \ll ( \lambda/
\paramone)^r $
\item  for each $j$, there is at most $O(1)$ values of $k$ for which
$\tau_j\tau_k\neq 0$ (independent of $\lambda$). In particular for
each point $t$ there is at most $O(1)$ values of $j$ so that
$\tau_j(t)\neq 0$.
\end{enumerate}

 Let $J_0\subseteq J$ be the set of indices $j$ for which $f$ has a
 sign change on $\supp \tau_j$. Since for each point $t$ there is at
 most  $O(1)$ values of $j$ for which $t\in \supp \tau_j$, we have
 \begin{equation}
 \#\mbox{ sign changes of }f \gg  \#J_0
 \end{equation}
 so that a lower bound for $\#J_0$ gives a lower bound for the
 number of sign changes of $f$.

If $j\notin J_0$ then $f\tau_j$
 does not change sign and hence
 \begin{equation}
 \int |f|\tau_j = |\int f\tau_j|, \quad
 j\notin J_0
 \end{equation}
 Therefore
 \begin{equation}
 \int|f| = \sum_j \int|f|\tau_j = \sum_{j\notin J_0} | \int f\tau_j|
 +\int |f|\sum_{j\in J_0} \tau_j
 \end{equation}
We will show that
\begin{equation}\label{bd on sum J_0}
\int |f|\sum_{j\in J_0}\tau_j \ll (\frac{\# J_0
\paramone}{\lambda})^{1/2}
\end{equation}
and that
\begin{equation}\label{bd on sum not in J0}
\sum_{j\notin J_0}|\int  f \tau_j | \ll \paramone^{-1/3}
\end{equation}
so that
\begin{equation}
\int|f| \ll (\frac{\# J_0
\paramone}{\lambda})^{1/2} + \paramone^{-1/3}
\end{equation}
Taking $\paramone^{-1/3} =\delta \int|f|$ with $\delta>0$
sufficiently small gives
\begin{equation}
\#J_0 \gg \lambda (\int|f|)^5
\end{equation}
which proves Theorem~\ref{thm N in terms of m int}.
Note  that our choice of $\paramone$ indeed satisfies our requirements, indeed $1\ll \paramone \ll \lambda^{o(1)}$ since
$\int|f|\ll (\int|f|^2)^{1/2} \ll ||F||_2\approx 1$,
by the upper bound in the  uniform $L^2$-restriction theorem \cite{BRCRAS}, and  $\int |f| \gg \lambda^{-o(1)}$
from the lower bound in the uniform $L^2$-restriction theorem, see \eqref{bd for m}.

\subsection{Proof of \eqref{bd on sum J_0}}
 By Cauchy-Schwarz,
\begin{equation}
\int |f|\sum_{j\in J_0} \tau_j \leq ||f||_2 \{\int (\sum_{j\in J_0}
\tau_j)^2 \}^{1/2} = ||f||_2 \{\sum_{j,k\in J_0} \int \tau_j \tau_k
\}^{1/2}
\end{equation}
By the restriction upper bound of \cite{BRCRAS}, $||f||_2\ll 1$.
Given $j$, we have $\int \tau_j\tau_k=0$ except for $O(1)$ indices
$k$ (independent of $j$), including $k=j$, and for such $k$ we have
\begin{equation}
\int \tau_j\tau_k \leq \frac 12(\int \tau_j^2 + \int \tau_k^2) \leq
\max_k \int \tau_k^2
\end{equation}
Since
\begin{equation}
\int \tau_k^2\leq \int_{\supp \tau_k} 1 \ll
\frac{\paramone}{\lambda}
\end{equation}
we obtain
\begin{equation}
\int (\sum_{j\in J_0} \tau_j)^2 \ll \#J_0 \max_k \int \tau_k^2 \ll
\frac{\# J_0 \paramone}{\lambda}
\end{equation}
and hence
\begin{equation}
\sum_{j\in J_0}\int |f|\tau_j \ll (\frac{\# J_0
\paramone}{\lambda})^{1/2}
\end{equation}

\subsection{Proof of \eqref{bd on sum not in J0}}
Our goal is to show that
\begin{equation}
\sum_{j\notin J_0}  \left| \int f(t)\tau_j(t) dt \right|
\end{equation}
 is small.

Let $\prho> 0$ be a (small) parameter, $\lambda^{-o(1)} <\prho<1$
 and $0\leq \theta(x) \leq 1$ a
smooth, even function so that $\theta(x)=1$ if $|x|<1$,
$\theta(x)=0$ for $|x|>2$ and set $\theta_\prho(x) = \theta(\frac
x\prho)$.

Write $f=f_0+f_1$ where
\begin{equation}
f_0(t) = \sum_{\mu\in \vE} a_\mu
\theta_\prho(\phi'_\mu(t))e^{i\lambda\phi_\mu(t)}
\end{equation}
and
\begin{equation}
f_1(t) = \sum_{\mu\in \vE} a_\mu
(1-\theta_\prho)(\phi'_\mu(t))e^{i\lambda\phi_\mu(t)}
\end{equation}
Thus in the Fourier expansion of $f_1$, none of the phase functions
$\phi_\mu$ have a critical point in the support of $(1-\theta_\prho)\phi'_\mu$,
in fact they satisfy $|\phi'_\mu(t)| \geq \prho$.

We have
\begin{equation}
\sum_{j\notin J_0} \left| \int f(t)\tau_j(t) dt \right| \leq \sum_{j\notin J_0} \left| \int
f_0(t)\tau_j(t) dt \right| + \sum_{j\notin J_0} \left| \int f_1(t)\tau_j(t) dt
\right|
\end{equation}
 We will show that
\begin{equation}\label{terms with f0}
\sum_{j\notin J_0} \left| \int f_0(t)\tau_j(t) dt \right| \ll \prho^{1/2}
\end{equation}
and
\begin{equation}\label{terms with f1}
\sum_{j\notin J_0} \left| \int f_1(t)\tau_j(t) dt \right| \ll \frac 1{\paramone
\prho}
\end{equation}
which gives
\begin{equation}
\sum_{j\notin J_0} |\int f\tau_j| \ll \prho^{1/2} +\frac{1}{\paramone\prho}
\end{equation}
Choosing $\paramone = \prho^{-3/2}$ gives
\begin{equation}
\sum_{j\notin J_0} |\int f\tau_j| \ll \paramone^{-1/3}
\end{equation}
proving \eqref{bd on sum not in J0}.

\subsection{Proof of \eqref{terms with f0}}
We have
\begin{equation}
\sum_{j\notin J_0}\left| \int f_0(t)\tau_j(t) dt \right| \leq \int|f_0(t)|dt \ll
  ||f_0||_2
\end{equation}
and hence \eqref{terms with f0} follows from:
\begin{lemma}\label{lem |f_0|}
\begin{equation}
||f_0||_2 \ll \prho^{1/2}
\end{equation}
\end{lemma}
\begin{proof}
We wish to apply Lemma~\ref{bilinear} with $h_\mu =
\theta(\frac{\phi_\mu'}{\prho})$. We clearly have
$||h_\mu||_\infty\leq 1$. Moreover,
\begin{equation}
||h_\mu||_2^2 \leq \int \theta(\frac{\phi'}{\prho}) \leq \length\{t:
|\phi'_\mu(t)|<2\prho\}
\end{equation}
and hence $||h_\mu||_2^2\ll \prho$ by Lemma~\ref{lem length Bmu}.
Likewise
\begin{equation}
||h'_\mu||_1  = \int |\theta'(\frac{\phi'_\mu}{\prho})|\frac
{|\phi_\mu''|} {\prho}  \leq \frac{K_{\max}|\theta'|_\infty}{\prho}
\length\{t: |\phi'_\mu(t)|<2\prho\}
\end{equation}
and hence $||h'_\mu||_1\ll 1$. Inserting into Lemma~\ref{bilinear}
gives
\begin{equation}
||f_0||_2^2 \ll \prho + \frac{\#\vE}{\lambda^{1/6}}
\end{equation}
which gives our claim provided $\prho \gg \lambda^{-o(1)}$.
\end{proof}

\subsection{Proof of \eqref{terms with f1}}

We expand and integrate by parts
\begin{equation}
\int f_1\tau_j =\frac 1{i\lambda} \int f_2\tau_j + \frac
1{i\lambda}\int f_3\tau'_j
\end{equation}
where
\begin{equation}
f_2 = \sum_{\mu\in \vE} a_\mu
\left(\frac{1-\theta_\prho(\phi'_\mu)}{\phi'_\mu} \right)'
e^{i\lambda \phi_\mu}
\end{equation}
and
\begin{equation}
f_3 =\sum_{\mu\in \vE} a_\mu
\frac{1-\theta_\prho(\phi'_\mu)}{\phi'_\mu}  e^{i\lambda \phi_\mu}
\end{equation}
Hence
\begin{equation}
\begin{split}
\sum_{j\notin J_0}|\int f_1 \tau_j| &\leq \frac 1\lambda \left( \int |f_2|\sum_{j\notin J_0}
\tau_j  + \int |f_3|  \sum_{j\notin J_0}|\tau'_j| \right) \\
&\leq \frac 1\lambda ||f_2||_2  + \frac 1\lambda|| f_3||_2
\left\{\int(\sum_{j\notin J_0} |\tau'_j|)^2\right\}^{1/2}
\end{split}
\end{equation}
We have
\begin{equation}
\int(\sum_{j\notin J_0} |\tau'_j|)^2  = \sum_{j\notin J_0}\sum_{k\notin J_0}\int|\tau'_j\tau'_k| \ll
\sum_{j\notin J_0} \int (\tau'_j)^2
\end{equation}
since for each $j$, there are only $O(1)$ values of $k$ for which
$\tau'_j\tau'_k\neq 0$. Hence
\begin{equation}
\int(\sum_{j\notin J_0} |\tau'_j|)^2  \ll\sum_{j\notin J_0} \int (\tau'_j)^2 \ll (\frac
\lambda{\paramone})^2
\end{equation}
Therefore
\begin{equation}
\sum_{j\notin J_0}|\int f_1 \tau_j| \ll \frac 1\lambda ||f_2||_2  + \frac
1{\paramone} || f_3||_2
\end{equation}

Using Lemma~\ref{bilinear}
we find
\begin{equation}
||f_2||_2 \ll \frac 1{\prho^2}\;,\quad || f_3||_2 \ll \frac 1{\prho}
\end{equation}
once we note that
\begin{equation}
|\frac{1-\theta_\prho(\phi'_\mu)}{\phi'_\mu}|  \ll \frac 1{\prho}
\end{equation}
\begin{equation}
|\left\{\frac{1-\theta_\prho(\phi'_\mu)}{\phi'_\mu}\right\}'| \ll
\frac 1{\prho^2}
\end{equation}
and, using \eqref{bound on phi'''}, that
\begin{equation}
|\left\{\frac{1-\theta_\prho(\phi'_\mu)}{\phi'_\mu}\right\}''| \ll
\frac 1{\prho^3}
\end{equation}
(we assume throughout that $\prho\gg \lambda^{-o(1)}$).  This gives
\begin{equation}
\sum_{j\notin J_0}|\int f_1 \tau_j| \ll \frac 1\lambda \frac 1{\prho^2} +\frac
1{\paramone}\frac 1{\prho} \ll \frac 1{\paramone \prho}
\end{equation}
proving \eqref{terms with f1}.



%




\section{Relating $L^1$ and $L^4$ restriction theorems}
\label{sec:reduction to L4}



We briefly explain the relation between  $L^1$ and $L^4$ restriction
theorems given in \eqref{L1 in terms of L4 v2 int}, namely
\begin{equation}\label{L1 in terms of L4 v2}
||F||_{L^1(\curve)}\gg_\curve \frac {||F||_2^3}
{||F||^{2}_{L^4(\curve)}}
\end{equation}

By Cauchy-Schwarz, $\int_\curve|F| \ll ||F||_{L^2(\curve)}$ and by
the upper bound in the $L^2$-restriction theorem \cite{BRCRAS} we
have $ ||F||_{L^2(\curve)}\ll ||F||_2 $ so that
\begin{equation}
  \int_\curve |F| \ll ||F||_2
\end{equation}
As for lower bounds,  we certainly have $\int_\curve |F|^2 \leq
||F||_\infty \int_\curve |F|$ and combining the lower bound in the
$L^2$-restriction theorem \cite{BRCRAS}, $\int_\curve |F|^2\gg
||F||_2^2$ with the upper bound on the $L^\infty$ norm $
||F||_\infty \leq \sqrt{\#\vE}||F||_2/2\pi$ (see \eqref{Linfty})
we obtain
\begin{equation}\label{bd for m}
\frac 1{||F||_2} \int_\curve |F|  \gg \frac 1{\sqrt{\#\vE}}
\end{equation}

We want  to improve the bound \eqref{bd for m} for $\int_\curve|F|$.
To start with, we use interpolation (log-convexity of the $L^p$
norm) to give a lower bound for $
||F||_{L^1(\curve)}=\int_\curve|F|$ in terms of the $L^2$ and $L^4$
norms on the curve:
\begin{equation} \label{L1 in terms of L4 on C}
||F||_{L^2(\curve)} \leq ||F||_{L^1(\curve)}^{1/3} \cdot
||F||_{L^4(\curve)}^{2/3}
\end{equation}
which improves on \eqref{bd for m} as it does not contain any
component which is a-priori unbounded in $\lambda$.


Inserting the uniform $L^2$ restriction lower bound
$||F|||_{L^2(\curve)}\gg ||F||_2$ of  \cite{BRCRAS} into \eqref{L1
in terms of L4 on C} gives \eqref{L1 in terms of L4 v2}
as claimed.



\section{An upper bound on the restriction $L^4$ norm: Proof of
Theorem~\ref{thm L4 int}} \label{sec: the L4 norm}

The aim of this section is to reduce getting a uniform upper bound
for the  $4$-th moment $\int_\curve |F|^4$, to counting lattice
points in arcs of length $\sqrt{\lambda}$ by showing that
\begin{equation}
\int_\curve |F|^4  = \int |f(t)|^4 dt \ll B_\lambda
\end{equation}
where as in \eqref{def of B},
\begin{equation}
B_\lambda=\max_{|x|=\lambda} \#\{\xi\in \vE: |x-\xi|\leq
\sqrt{\lambda}\}
\end{equation}

\subsection{Computing $\int|f|^4$}
Recall
\begin{equation}
f(t) = \sum_{\mu} a_\mu e^{i\langle \mu,\gamma(t) \rangle}
\end{equation}
We may break up $f$ into $O(1)$ terms, each the sum over frequencies
$\mu$ lying in an arc of size $\lambda/100$. By the triangle
inequality, it suffices to prove the restriction $L^4$ bound for
such $f$, and from now on we assume that $f$ is of this form.

In order to compute the $4$-th moment $\int |f|^4$, write
\begin{equation}\label{expand f^2}
f(t)^2 = \sum_{\mu,\nu} a_\mu a_\nu e^{i\langle
\mu+\nu,\gamma(t)\rangle} =   \sum_\mu a_\mu a_{-\mu} +\sum_{0\neq
z\in \mathcal Z} b_ze^{2i\langle z, \gamma(t)\rangle}
\end{equation}
where  for a median $z=(\mu+\nu)/2\in \mathcal Z$ (see
\S~\ref{sec:medians}), we set $b_z = 2a_\mu a_\nu$. The assumption
that all the frequencies $\mu$ lie in an arc of size $\lambda/100$
implies that the medians $z\in \mathcal Z$ appearing in
\eqref{expand f^2} satisfy $|z|>\lambda/2$, and that $a_\mu
a_{-\mu}=0$ for all $\mu$.
 Observe that
\begin{equation}
\sum_{0\neq z\in \mathcal Z}|b_z|^2 \ll (\sum_\mu |a_\mu|^2)^2 =
\frac 1{(2\pi)^4} ||F||_2^4
\end{equation}


Hence we can we write
\begin{equation}
f(t)^2 = \tg_0(t) + \tg(t) 
\end{equation}
with
\begin{equation}
\tg_0(t) = \sum_{0<\Delta(z)\leq \sqrt{\lambda} } b_z e^{2i\langle
z, \gamma(t)\rangle}
\end{equation}
and
\begin{equation}
\tg(t) = \sumstar_{z} b_z e^{2i\langle z, \gamma(t)\rangle}
\end{equation}
where we denote
\begin{equation}
\sumstar_z:=\sum_{\substack{z\in \mathcal Z\\|z|\geq \lambda/2  \\
|\Delta(z)|>\sqrt{\lambda}}}
\end{equation}
 Therefore
 \begin{equation}
 ||f||_4  = ||f^2||_2^{1/2}\leq \left( ||g_0||_2 + ||g||_2 
) \right)^{1/2}
 \end{equation}
so that it suffices to show
\begin{equation}
||\tg_0||_2^2 \ll B_\lambda  ||b||^2\;,  \qquad  ||\tg||_2^2 \ll
B_\lambda ||b||^2
\end{equation}
where $b=(b_z)\in \C^{\mathcal Z}$.

By   Lemma~\ref{osc int curve},  if $z\neq w$ then
\begin{equation}
\int e^{2i\langle z-w,\gamma(t)\rangle} dt \ll \frac 1{|z-w|^{1/2}}
\end{equation}
 and since the integral is trivially bounded by $\ll 1$, we can write this for any pair $z, w\in \mathcal Z$ as
\begin{equation}
\int e^{2i\langle z-w,\gamma(t)\rangle} dt \ll \frac
1{|z-w|_+^{1/2}}
\end{equation}
where
\begin{equation}
|z|_+ = \max \left( 1, |z| \right)
\end{equation}
Therefore
\begin{equation}
\int |\tg|^2 \ll \sumstar_z\sumstar_w
\frac{|b_z||b_w|}{|z-w|_+^{1/2}}
\end{equation}
Moreover, we may restrict the sum to $|w-z|<\lambda^\epsilon$ at a
cost of $O(\lambda^{-\epsilon/2}||b||^2\#\mathcal Z) = o(||F||_2^4)$ since
$\#\mathcal Z\ll \lambda^{o(1)}$. Denoting by
\begin{equation}
\sumstar_{z,w} := \sumstar_z\sumstar_{w:\; |z-w|<\lambda^\epsilon}
\end{equation}
we have found that
\begin{equation}
||\tg||^2 \ll \sumstar_{z,w}\frac {|b_z b_w|}{|z-w|_+^{1/2}}
\end{equation}
and likewise
\begin{equation}
||\tg_0||^2 \ll \sum_{\substack{0<\Delta(z),\Delta(w)\leq
\sqrt{\lambda}\\|z-w|<\lambda^\epsilon}}\frac {|b_z
b_w|}{|z-w|_+^{1/2}}
\end{equation}
Thus we see that it suffices to show:
\begin{proposition}\label{prop bilinear schur}
Let $b=(b_z)\in \C^{\mathcal Z}$. Then
\begin{equation} \label{prop 5.1bis}
\sum_{\substack{0<\Delta(z),\Delta(w)<\sqrt{\lambda}\\|z-w|<\lambda^\epsilon}}\frac
{|b_z b_w|}{|z-w|_+^{1/2}} \ll B_\lambda ||b||^2 \tag{i}
\end{equation}
and
\begin{equation} \label{bilinear schur}
\sumstar_{z,w}\frac {|b_z b_w|}{|z-w|_+^{1/2}} \ll B_\lambda ||b||^2
\tag{ii}
\end{equation}
\end{proposition}


\subsection{Proof of Proposition~\ref{prop bilinear schur} \eqref{prop
5.1bis}}
 By  Schur's test,
\begin{equation}
\sum_{\substack{0<\Delta(z),\Delta(w)\leq \sqrt{\lambda}\\
|z-w|<\lambda^\epsilon}} \frac{|b_z b_w|}{|z-w|_+^{1/2}}\leq
\max_{0<\Delta(z)\leq \sqrt{\lambda}} \sum_{\substack{0< \Delta(w)\leq \sqrt{\lambda}\\
|z-w|<\lambda^\epsilon} }\frac{1}{|z-w|_+^{1/2}} ||b||^2
\end{equation}
and so it suffices to show  that
\begin{equation}
\sum_{\substack{0< \Delta(w)\leq \sqrt{\lambda}\\
|z-w|<\lambda^\epsilon} }\frac{1}{|z-w|_+^{1/2}} \ll B_\lambda
\end{equation}
Replacing $|z-w|_+$ by $1$ we are reduced to showing that
\begin{equation}\label{counting S0}
\#\{0< \Delta(w)\leq \sqrt{\lambda},\quad  |z-w|<\lambda^\epsilon \}
\ll B_\lambda
\end{equation}

  We have
\begin{equation}
\mu_+(w) - \mu_+(z) =(w-z) +
(\Delta(w)-\Delta(z))\frac{z^\perp}{|z|} + \Delta(w)(
\frac{w^\perp}{|w|}-\frac{z^\perp}{|z|})
\end{equation}
Since $\Delta(z),\Delta(w)<\sqrt{\lambda}$ we have $|z|,|w|\sim
\lambda$, and hence
\begin{equation}
|\frac{w^\perp}{|w|}-\frac{z^\perp}{|z|}|\ll \frac{|z-w|}\lambda\ll
\lambda^{-1+\epsilon}
\end{equation}
Thus
\begin{equation}
|\mu_+(w) - \mu_+(z)| \ll \sqrt{\lambda}
\end{equation}
By Lemma~\ref{stability lemma}
  we see that there are at most $O(B_\lambda)$
possibilities for $w$. This proves \eqref{counting S0}. \qed

\subsection{A dyadic subdivision}
We turn to the proof of part \eqref{bilinear schur} of
Proposition~\ref{prop bilinear schur}. For $K\geq 1$, let
\begin{equation}
S_K = \{z\in  \mathcal Z, \quad K\sqrt{\lambda}\leq \Delta(z)
<2K\sqrt{\lambda} \}
\end{equation}
We write
\begin{equation}
\sumstar_{z,w}\frac {|b_z b_w |}{ |z-w|_+^{1/2}}   = \sum_{K,L
\;\rm{ dyadic}}\langle A_{K,L}b^{(K)}, b^{(L)} \rangle\;,
\end{equation}
 the sum over $K=2^k$, $L=2^\ell$, with
\begin{equation}
 \langle A_{K,L}b^{(K)}, b^{(L)} \rangle= \sumstar_{\substack{w\in S_L\\ z\in
S_K}} \frac {|b_z b_w |}{|z-w|_+^{1/2}}\;,
\end{equation}
where $b^{(K)} = (|b_z|)_{z\in S_K}$,  $b^{(L)} = (|b_w|)_{w\in
S_L}$, and $A_{K,L}:\C^{S_K}\to \C^{S_L}$ is the matrix
\begin{equation}
A_{K,L} = ( \frac 1{|z-w|_+^{1/2}})_{z\in S_K, w\in S_L}
\end{equation}
with zeros whenever one of the conditions
$\Delta(z),\Delta(w)>\sqrt{\lambda}$, $|z|,|w|>\lambda/2$  or
$|z-w|<\lambda^\epsilon$ is violated.

 We use Schur's test for the operator norm:
 \begin{equation}
 ||A_{K,L}||_{2\to 2}\leq
||A_{K,L}||_{1\to 1}^{1/2} \cdot ||A_{K,L}^*||_{1\to 1}^{1/2}
\end{equation}
 to bound
\begin{equation}
  |\langle A_{K,L}b^{(K)}, b^{(L)} \rangle |\leq ||A_{K,L}||_{1\to 1}^{1/2} \cdot ||A_{K,L}^*||_{1\to 1}^{1/2}
\cdot ||b^{(K)}|| \cdot ||b^{(L)}||
\end{equation}
where $||b^{(K)}||$ is the $\ell^2$-norm. We will show
\begin{proposition}\label{prop bounds on AKL}
For $K\leq L$,
\begin{equation}
||A_{K,L}||_{1\to 1} \ll  B_\lambda
\end{equation}
and
\begin{equation}
 ||A_{K,L}^*||_{1\to 1} \ll \frac KL B_\lambda
\end{equation}
\end{proposition}
Therefore
 \begin{equation}
 \begin{split}
\sumstar_{z,w}\frac {|b_z b_w|}{|z-w|_+^{1/2}}  &\ll
 B_\lambda \sum \sum_{K,L\;
\rm{dyadic}} (\frac{\min(K,L)}{\max(K,L)})^{1/2} ||b ^{(K)}|| \cdot ||b^{(L)}||\\
&\ll B_\lambda  \left\{\max_K \sum_{\substack {L=2^\ell \; \rm{dyadic}\\
L\geq K}}(\frac KL)^{1/2} \right\} \cdot \sum_K ||b^{(K)}||^2
\end{split}
\end{equation}
Since $ \sum_K ||b^{(K)}||^2 = ||b||^2 $ and
\begin{equation}
\sum_{\substack {L=2^\ell \; \rm{dyadic}\\
L\geq K}}(\frac KL)^{1/2} \ll 1
\end{equation}
 we will have proved part \eqref{bilinear schur} in Proposition~\ref{prop bilinear schur}. \qed

\subsection{Proof of Proposition~\ref{prop bounds on AKL}}
By Schur's test,
\begin{equation}
||A_{K,L}||_{1\to 1}\leq\max_{z\in S_K} \sumstar_{\substack{w\in
S_L\\|z-w|<\lambda^\epsilon}} \frac
1{|z-w|_+^{1/2}}  
\end{equation}
and
\begin{equation}
 ||A_{K,L}^*||_{1\to 1}  \leq \max_{w\in S_L} \sumstar_{\substack{z\in S_K\\|z-w|<\lambda^\epsilon}}
 \frac 1{|z-w|_+^{1/2}}
\end{equation}



\begin{lemma}\label{lem number of w}
Let $z\in S_K$. Then
\begin{equation}\label{number of w}
\#\{w\in S_L: |w-z|<\lambda^\epsilon\} \ll L   B_\lambda
\end{equation}
\end{lemma}
\begin{proof}
For  $0\leq l\leq L-1$, set
\begin{equation}\label{SLl}
S_{L,l} = \{w\in S_L:\quad (L+l)\sqrt{\lambda}\leq \Delta(w)
<(L+l+1)\sqrt{\lambda} \}
\end{equation}
We show that for $z\in S_K$, $w\in S_{L,l}$ and
$|z-w|<\lambda^\epsilon$, we have
\begin{equation}\label{dist muw}
 |\mu_+(w)-v| \ll \sqrt{\lambda}
\end{equation}
where
\begin{equation}
v = \mu_+(z) + \left((L+l) \sqrt{\lambda} -\Delta(z)\right)
\frac{z^\perp}{|z|}
\end{equation}
 By Lemma~\ref{stability lemma} we see that there are at most $O(B_\lambda)$ possibilities for $w$ in
$S_{L,l}$ subject to $|w-z|<\lambda^\epsilon$:
\begin{equation}\label{bound ****}
\#\{w\in S_{L,l}: |w-z|<\lambda^\epsilon\} \ll B_\lambda
\end{equation}
Since
\begin{equation}
S_L = \bigcup_{l=0}^{L-1} S_{L,l}
\end{equation}
we find
\begin{equation}
\#\{w\in S_L:|w-z|<\lambda^\epsilon\} \ll L B_\lambda
\end{equation}
as claimed.

To prove \eqref{dist muw}, we use  \eqref{eq for mu} to get
\begin{equation}
\mu_+(w)-\mu_+(z) = w-z + \Delta(w) \frac{w^\perp}{|w|} -\Delta(z)
\frac{z^\perp}{|z|}
\end{equation}
Note that $|w^\perp| = |w|$ and since $|z|\geq \lambda/2$ and
$|z-w|<\lambda^\epsilon$ then $|w|=|z|+O(\lambda^\epsilon)\gg
\lambda$ and so
\begin{equation}
\frac{w^\perp}{|w|} =\frac{z^\perp}{|z^\perp|} +
O(\lambda^{-1+\epsilon})
\end{equation}
Hence
\begin{equation}
\begin{split}
\mu_+(w)-\mu_+(z)   &=O(\lambda^\epsilon)
+\Delta(w)(\frac{z^\perp}{|z^\perp|} + O(\lambda^{-1+\epsilon})) -
\Delta(z)\frac{z^\perp}{|z^\perp|}\\
&=(\Delta(w)-\Delta(z)) \frac{z^\perp}{|z^\perp|} + O(
\lambda^\epsilon)
\end{split}
\end{equation}
Writing 
$\Delta(w) = (L+l+\theta)\sqrt{\lambda}$ with $0\leq \theta<1$ we
find
\begin{equation}
\mu_+(w)-\mu_+(z)  = \left((L+ l)\sqrt{\lambda}-\Delta(z)\right)
\frac{z^\perp}{|z^\perp|} + \theta\sqrt{\lambda}
\frac{z^\perp}{|z^\perp|} + O(\lambda^\epsilon)
\end{equation}
proving \eqref{dist muw}.
\end{proof}

We now give a lower bound for the difference of medians  $|z-w|$:

\begin{lemma}\label{bound z-w}
Let $K\leq L$, $w\in S_{L,l}$, $z\in S_{K,k}$,
$0<|w-z|<\lambda^\epsilon$. Then
\begin{enumerate}
\item If $2K<L$ then $|z-w|\gg L^2$;

\item If $K=L/2$ and $l\neq 0$ then $|z-w|\gg Ll$ ;

\item If $K=L$ and $l\neq k, k\pm 1$ then $|z-w|\gg L|l-k|$.

\end{enumerate}
\end{lemma}

\begin{proof}
 To bound $|z-w|$, the condition $K\leq L$ allows us to assume $|z|\geq |w|$.
Then
\begin{equation}
\begin{split}
|z-w| &\geq |z|-|w| = \frac{\lambda^2-|w|^2}{\lambda+|w|} -
\frac{\lambda^2-|z|^2}{\lambda+|z|}\\
& \geq \frac{\Delta(w)^2}{\lambda+|z|} -
\frac{\Delta(z)^2}{\lambda+|z|}
\\
 &\geq \frac{\Delta(w)^2-\Delta(z)^2}{2\lambda
}\\
&\gg \frac{L}{\sqrt{\lambda}} |\Delta(w)-\Delta(z)| \gg L(L+l-K-k-1)
\end{split}
\end{equation}

If $K\leq L/4$ then $L(L+l-K-k-1)\geq L(L-2K)\geq L^2/2$. If $K=L/2$
then $L|L+l-K-k-1|\geq L(L+l-2K)=Ll$, useful if $l\neq 0$. Finally
if $K=L$ then $L(L+l-K-k-1|=L|l-k-1|\geq \frac 12 L|l-k|$ if
$|l-k|\geq 2$, the exceptional cases being  $l=k,k\pm 1$.
%
%
%
\end{proof}

\subsection{Bounding $||A_{K,L}||_{2\to 2}$}

We want to show that, if $K\leq L$, then
\begin{equation}
||A_{K,L}||_{1\to 1}\leq \max_{z\in S_K} \sum_{\substack{w\in S_L\\|z-w|<\lambda^\epsilon}} \frac
1{|z-w|_+^{1/2}}  \ll  B_\lambda
\end{equation}
and
\begin{equation}
 ||A_{K,L}^*||_{1\to 1}  \leq \max_{w\in S_L} \sum_{\substack{z\in S_K\\|z-w|<\lambda^\epsilon}} \frac 1{|z-w|_+^{1/2}} \leq \frac KL B_\lambda
\end{equation}

We assume first that $2K<L$, that is $K=2^k$, $L=2^\ell$ with $\ell
\geq k+2$. Let
\begin{equation}
S_L(z,\lambda^\epsilon) = \{w\in S_L:|w-z|<\lambda^\epsilon\}
\end{equation}
 We have
\begin{equation}
\begin{split}
||A_{K,L}||_{1\to 1}&\leq  \max_{z\in S_K} \sum_{ w\in S_L(z,\lambda^\epsilon) } \frac
1{|z-w|_+^{1/2}} \\
&\leq \max_{z\in S_K}\frac {\#S_L(z,\lambda^\epsilon)}{\min_{w\in
S_L(z,\lambda^\epsilon)}|z-w|_+^{1/2}}
\end{split}
\end{equation}
According to Lemma~\ref{lem number of w},
\begin{equation}
\#S_L(z,\lambda^\epsilon)\ll L B_\lambda
\end{equation}
 and by   Lemma~\ref{bound z-w}, if   $2K<L$ then
\begin{equation}
\min_{w\in S_L(z,\lambda^\epsilon)}|z-w| \gg L^2
\end{equation}
 Hence we find (for $2K<L$)
\begin{equation}
||A_{K,L}||_{1\to 1} \ll B_\lambda
\end{equation}

Arguing in the same way with the roles of $L$ and $K$ reversed
gives, for $2K<L$, that
\begin{equation}
 ||A_{K,L}^*||_{1\to 1}    \leq \max_{z\in S_K}\frac {\#S_K(z,\lambda^\epsilon)}{\min_{w\in
S_K(z,\lambda^\epsilon)}|z-w|_+^{1/2}} \leq \frac {K B_\lambda}L
\end{equation}

\subsection{The cases $K=L/2,L$}
It remains to deal with the case $K=L/2$ and $K=L$. We use the
decomposition $S_L = \bigcup_{l=0}^{L-1} S_{L,l}$    in \eqref{SLl}
to write
\begin{equation}
\begin{split}
 \sum_{ w\in S_{L }(z,\lambda^\epsilon)} \frac 1{ |z-w|_+^{1/2}} &\ll
\sum_{l=0}^{L-1}\sum_{w\in S_{L,l}(z,\lambda^\epsilon)} \frac 1{|z-w|_+^{1/2}}\\
&\ll \sum_{l=0}^{L-1}\frac {\#S_{L,l}(z,\lambda^\epsilon)} {
\min_{w\in S_{L,l}(z,\lambda^\epsilon)}|z-w|_+^{1/2}}\\
&\ll B_\lambda \sum_{l=0}^{L-1}\left(\min_{w\in
S_{L,l}(z,\lambda^\epsilon)}|z-w|_+^{1/2}  \right)^{-1}
\end{split}
\end{equation}
where we have used \eqref{bound ****}. Applying Lemma~\ref{bound z-w} gives for $K=L/2$
\begin{equation}
\sum_{  w\in S_{L }(z,\lambda^\epsilon) } \frac 1{
|z-w|_+^{1/2}}\ll B_\lambda\left( 1+\sum_{l=1}^{L-1} \frac
1{(Ll)^{1/2}}   \right) \ll B_\lambda
\end{equation}
and if $K=L$ and $z\in S_{L,k}$ we get
\begin{equation}
\sum_{  w\in S_{L }(z,\lambda^\epsilon) } \frac 1{|z-w|_+^{1/2}}\ll
B_\lambda \left(\sum_{\substack{0\leq l\leq L-1\\|l-k|\geq 2}} \frac
1{L^{1/2}|l-k|^{1/2}}  + O(1) \right) \ll B_\lambda
\end{equation}

Thus we find that  $||A_{K,L}^*||_{1\to 1} , ||A_{K,L}^*||_{1\to 1}
\ll B_\lambda$ for  $K=L/2,L$, concluding the proof of
Proposition~\ref{prop bounds on AKL}. \qed

\section{Exceptions on the sphere and the torus}\label{Sec geod}
\subsection{Nodal intersections with geodesics on the torus}

We conclude by pointing out that no lower bound on $N_{F,\curve}$ is
possible without the assumption on non-vanishing of the curvature of
$\curve$, that is when $\curve$ is flat.

When $\curve$ is a segment of a closed geodesic on the torus, there
are arbitrarily large eigenvalues $\lambda$ for which there are
eigenfunctions $F_\lambda$ vanishing identically on $\curve$, that
is for which $\curve \subset \mathcal N_{F_\lambda}$. Indeed, if the
curve is a segment of the rational line $px+qy=c$, with $(0,0)\neq
(p,q)\in \Z^2$ then taking $F_n(x,y) = \sin n(qx-py-c)$,
$n=1,2,\dots$, gives an eigenfunction which has eigenvalue
$n^2(p^2+q^2)$ and which vanishes on the entire closed geodesic. See
\cite{BRINV} for further discussion of such "persistent components".

For the case when $\curve$ is a segment of an unbounded geodesic, we
claim that there are always arbitrarily large eigenvalues
$\lambda_k^2$ for which there is an eigenfunction $F_k$ with
$N_{F_k,\curve}=0$. To see this, take an irrational $\beta\notin
\Q$, and $\vec v_0\in \R^2$  and let $\curve$ be the irrational line
segment $\{\vec v_0+ t(1,-\beta): |t|<1\}$. Let $\vec n_k=(p_k,q_k)$
be a sequence of good rational approximations of $\beta$:
\begin{equation}
|\beta-\frac {p_k}{q_k}|<\frac 1{q_k^2}
\end{equation}
with $q_k\to +\infty$. Let
\begin{equation}
F_k(\vec x) = \cos(\vec n_k \cdot (\vec x-\vec v_0) )
\end{equation}
which is an eigenfunction with eigenvalue $\lambda_k^2
=p_k^2+q_k^2$. Then on $\curve$ we have
\begin{equation}
F_k(\vec v_0+ t(1,-\beta)) = \cos(t(p_k-q_k\beta))
\end{equation}
Since
\begin{equation}
|t(p_k-q_k\beta)|\leq |p_k-q_k\beta|<\frac 1{q_k}
\end{equation}
we see that
\begin{equation}
F_k(\vec v_0+ t(1,-\beta))  = 1+O(\frac 1{q_k^2})
\end{equation}
and so for $k\gg 1$, $F_k|_\curve$ has no zeros.

\subsection{The sphere}

On the sphere, a basis of eigenfunctions is provided by the
spherical harmonics. Lets restrict attention to {\em zonal}
spherical harmonics. They are of the form $Y_\ell^0 = P_\ell (\cos
\theta)$ where $\theta$ is the colattitude, and $P_\ell(x)$ are the
Legendre polynomials
\begin{equation}
 P_\ell(x) = \frac 1{2^\ell} \sum_{j=0}^{\lfloor \ell/2 \rfloor} (-1)^j \binom{\ell}{j}\binom{2\ell-2j}{\ell-2j} x^{\ell-2j}
\end{equation}
which are orthogonal polynomials on  the interval $[-1,1]$. The
nodal set of the zonal spherical harmonic $Y_\ell^0$ is the union of
the parallels $\theta=\theta_{\ell,j}$, $j=1,\dots,\ell$ where
$x_{\ell,j} = \cos \theta_{\ell,j}$ are the zeros of the Legendre
polynomial $P_\ell(x)$.

Since $P_\ell(-x) = (-1)^\ell P_\ell(x)$, for odd $\ell$ we have
$P_\ell(0)=0$, and so we find that the zonal spherical harmonics
vanish on the equator $\curve(\pi/2)=\{\theta=\pi/2\}$ for odd
$\ell$, that is $N_{\curve(\pi/2),Y_\ell^0}=\infty$.

For other parallels $\curve(\theta_0)=\{\theta=\theta_0\}$,
$0<\theta_0<\pi/2$, we claim that there are infinitely many $\ell$
with $N_{\curve(\theta_0),Y_\ell^0}=0$. Thus even though the
parallels have nonzero curvature, no analogue for the lower bound of
of Theorem~\ref{thm lower bd for N} can hold on the sphere. To see
this, note that if $\cos \theta_0$ is not one of the countably many
zeros of the $P_\ell(x)$, then all the $Y_\ell^0$ never vanish
there. If $P_L(\cos\theta_0)=0$, then we claim that
$P_p(\cos\theta_0)\neq 0$ for all {\em prime} $p> L+1$. Indeed, when
$p>2$ is prime, Holt \cite{Holt} showed in 1912 that $P_p(x)/x$ are
irreducible over the rationals, and since $\deg P_\ell=\ell$, we
must have $\gcd(P_p(x)/x,P_L(x))=1$ and in particular they have no
common zeros.

\end{document}